\documentclass[12pt,a4paper,twoside]{amsart}
\usepackage{amsmath}
\usepackage{amssymb}
\usepackage{amsfonts}
\usepackage{enumerate}

\date{}

\def\tle{\triangleright}
\def\Diff{\text{Diff}}

\def\Aut{\text{Aut}}

\def\Isom{\text{Isom}}
\def\Sym{\text{Sym}}
\def\Out{\text{Out}}
\def\Inn{\text{Inn}}
\def\Homeo{\text{Homeo}}
\def\Hom{\text{Hom}}

\def\Z{\bold Z}

\def\Ga{\Gamma}
\def\bbr{\mathbb{F}}

\def\bbz{\mathbb{Z}}
\def\bbn{\mathbb{N}}

\theoremstyle{plain}
\newtheorem{thm}{Theorem}[section]
\newtheorem*{thm*}{Theorem}
\newtheorem{ingr}[thm]{Ingredient}

\newtheorem{prop}[thm]{Proposition}
\newtheorem*{prop*}{Proposition}

\def\beq{\begin{equation}}

  \def\ee{\end{equation}}

\theoremstyle{definition} 
\newtheorem{definition}[thm]{Definition}
\newtheorem*{definition*}{Definition}
\newtheorem*{thm1*}{Theorem A1}
\newtheorem*{thm2*}{Theorem A2}

\newtheorem*{conjecture*}{Conjecture}
\newtheorem{cor}[thm]{Corollary}

\newtheorem*{claim*}{Claim}

\newtheorem*{remark*}{Remark}
\def\bbz{\mathbb{Z}}

\def\bbr{\mathbb{R}}

\def\bbh{\mathbb{H}}
\def\be{\begin{equation}}
\def\ee{\end{equation}}

\def\a{\alpha}

\def\ol{\overline}

\theoremstyle{remark}  
\begin{document}

\title[Finite transformation groups]{A Trichotomy Theorem\\ for transformation groups\\ of locally symmetric manifolds\\ and topological rigidity}
\author[S. Cappell]{Sylvain Cappell}
\author[A. Lubotzky]{Alexander Lubotzky}
\author[S. Weinberger]{Shmuel Weinberger}
\maketitle


\baselineskip 16pt


\begin{abstract}  Let $M$ be a locally symmetric irreducible closed manifold of dimension $\ge 3$.  A result of Borel [Bo] combined with Mostow rigidity imply that there exists a finite group $G = G(M)$ such that any finite subgroup of $\Homeo^+(M)$ is isomorphic to a subgroup of $G$.  Borel [Bo] asked if there exist $M$'s with $G(M)$ trivial and if the number of conjugacy classes of finite subgroups of $\Homeo^+(M)$ is finite.  We answer both questions:
\begin{enumerate}
\item{} For every finite group $G$ there exist $M$'s with $G(M) = G$, and
\item{} the number of maximal subgroups of $\Homeo^+(M)$ can be either one, countably many or continuum and we determine (at least for $\dim M \neq 4$) when each case occurs.
\end{enumerate}
Our detailed analysis of (2) also gives a complete characterization of the topological local rigidity and topological strong rigidity (for dim$M\neq 4$) of proper discontinuous actions of uniform lattices in semisimple Lie groups on the associated symmetric spaces.
\end{abstract}

\section{Introduction}
A positive dimensional, oriented, closed manifold $M$ has a very large group of automorphisms (i.e., orientation preserving  self homeomorphisms).  In fact this group $\Homeo^+(M)$ is infinite dimensional.  But its finite subgroups are quite restricted.
In 1969, Borel showed (in a classic paper \cite{Bo} but which appeared only in 1983 in his collected works) that if $M$ is a $K(\pi, 1)$-manifold with fundamental group $\Ga=\pi_1 (M)$, whose center is trivial, then every finite transformation group $G$ in $\Homeo^+(M)$ is mapped injectively into the outer automorphism group $\Out(\Ga)$ by the natural map (or more precisely into the subgroup $\Out^+(\Ga)$ -- see \S 2 -- which has an index at most 2 in $\Out(\Ga)$).

Let now $M$ be a locally symmetric manifold  of the form $\Ga\backslash H / K$ when $H$ is a connected non-compact semisimple group with trivial center and with no compact factor and $\Gamma$ a torsion free uniform irreducible lattice in $H$.  In the situation in which strong rigidity holds (i.e.,  if $H$ is not locally isomorphic to $SL_2(\bbr)$), $\Out (\Ga)$ is a finite group, $G \le \Out^+(\Ga)$; in fact, $\Out^+(\Ga)  =  N_H (\Ga)/\Ga$ and it acts on $M$ as the group of (orientation preserving) self isometries $\Isom^+(M)$ of the Riemannian manifold $M$.  It follows now from Borel's theorem that every finite subgroup of $\Homeo^+(M)$ is isomorphic to a subgroup of one finite group, $G(M)= \Isom^+(M)$.

Borel ends his paper by remarking: ``The author does not know whether the finite subgroups of $\Homeo^+(M)$ form finitely many conjugacy classes, nor whether one can find a $\Ga$ with no outer automorphism."

The goal of the current paper is to answer these two questions.  For an efficient formulation of our results, let us make the following definition(s):

\begin{definition}
Let $G$ be a finite group.  An oriented manifold $M$ will be called $G$-exclusive (resp., $G$-weakly exclusive) if there is a faithful action of $G$ on $M$, so that $G$ can  be identified with a subgroup of $\Homeo^+(M)$ and if $F$ is any finite subgroup of $\Homeo^+(M)$, then $F$ is conjugate (resp., isomorphic) to a subgroup of $G$.
\end{definition}
Note that Borel's Theorem combined with strong rigidity implies that unless $H$ is locally isomorphic to $SL_2(\bbr)$, $M$ as above is always at least $\Isom^+(M)$-weakly exclusive.

We now claim:
\begin{thm}
For every finite group $G$ and every $3 \le n \in \bbn$, there exist infinitely many oriented closed hyperbolic manifolds $M=M^n(G)$ of dimension $n$, with $G\simeq \Isom^+(M)$ and when $n \neq 4$ these $M^n(G)$ are also $G$-exclusive.
\end{thm}

The  very special case $G = \{ e \}$ answers Borel's second question (where one can also deduce it from [BL]). Along the way it also answers the question of Schultz  \cite{Sc2}, attributed there to D. Burghelea, who asked whether  there exist asymmetric closed manifolds with degree one maps onto hyperbolic manifolds.
Our examples are even hyperbolic themselves.

The situation for dimension 2 is very different:
\begin{thm} For no group $G$, does there exist a $G$-weakly exclusive 2-dimensional closed manifold.
\end{thm}
In fact, for every closed, oriented  surface $\Sigma_g$, of genus $g, \,   \Homeo^+ (\Sigma)$ has more than one (but only finitely many) isomorphism classes of maximal finite subgroups, and this  number is unbounded as a function of $g$ - see Proposition 6.2.

As mentioned before, for $M$ as above with $\dim M\ge 3, M$ is always $G$-weakly exclusive.  But the $G$-exclusiveness shown in Theorem 1.2 is not the general phenomenon.  We can determine the situation in (almost) all cases. But first we need another definition.

\begin{definition}
For an automorphism $\varphi$ of a manifold $M$, denote by $Fix(\varphi)$ the fixed point set of $\varphi$ and for a subgroup $G \subseteq \Homeo^+(M)$, denote its {\it singular set}  $S(G) = \cup \{ Fix (\varphi) | \varphi \in G, \varphi\neq id \}$. If $M$ is an oriented Riemannian manifold, then we will call $S(\Isom^+(M))$ the {\it singular set} of $M$ and we denote it $S_M$.
\end{definition}

 We note that $\dim(M) - \dim(S_M)$ is always even, as we are only considering orientation preserving actions.

\medskip

Before stating our main theorem, let us recall that in our situation, i.e., when $M$ is locally symmetric, every finite subgroup of $\Homeo^+(M)$ is contained in a maximal finite subgroup.  We can now give a very detailed answer to Borel's first question.

\begin{thm}[Trichotomy Theorem]
Let $M = \Ga\backslash H/K$ a locally symmetric manifold as above, and assume $\dim M \neq 2$ or $4$.  Let  $G = \Isom^+ (M)$, so $G \cong N/\Ga$ where $N = N_H(\Ga)$.  Then one of the following holds:

\begin{enumerate}[{\rm(a)}]

\item $\Homeo^+(M)$  has a unique conjugacy class of maximal finite subgroups, all of whose members are conjugate to $\Isom^+(M)$.

    \item $\Homeo^+(M)$ has countably infinite many maximal finite subgroups, up to conjugacy \ or
    \item $\Homeo^+(M)$ has a continuum of  such subgroups (up to conjugacy).
    \end{enumerate}
These cases happen, if and only if the following hold, respectively:

\begin{enumerate}[{\rm (a)}]
\item \ {\rm (i)} $ S_M = \phi$, i.e., $\Isom^+(M)$ acts freely on $M$, or
\newline
\ \  {\rm (ii)}  the singular set $S_M$ is 0-dimensional and either $\dim(M)$ is divisible by 4 or all elements  of order 2 act freely.
\item $M$ is of dimension equal 2 (mod 4), the singular set is 0-dimensional and some element of order 2 has a non-empty fixed point set, or
\item the singular set $S_M$ is positive dimensional, i.e.,  $M$ has some non-trivial isometry with a positive dimensional fixed point set.
\end{enumerate}
\end{thm}

    The cases treated in Theorem 1.2 are with $N = N_H(\Ga)$ torsion free (see \S2), i.e.,  $S(G) = \phi$ where $G = \Isom^+(M)$, so we are in case (a)(i) and
     these manifolds $M$ are $\Isom^+(M)$-exclusive also by Theorem 1.5.

An interesting corollary of the theorem is that if $\Homeo^+(M)$ has only finitely many conjugacy classes of maximal finite subgroups then it has a unique one, the class of $\Isom^+(M)$, in contrast to Theorem 1.3.

In dimension 4 when the action has positive dimensional singular set, we do construct uncountably many actions.  If the singular set is finite, then we have countability, but we do not know whether/when this countable set of actions consists of a unique possibility. As a consequence, the following dichotomy holds in all dimensions:

\begin{cor} Let $M$ be as above with arbitrary dimension.  Then $\Homeo^+ (M) $ has an uncountable number of conjugacy classes of finite subgroups if and only if the singular set of $\Isom^+ (M)$ acting on $M$ is positive dimensional.
\end{cor}

The uniqueness in Theorem 1.5 fails in the smooth case (i.e.,  for $\Diff^+(M)$).  In that case, the number of conjugacy classes is always countable.  The boundary between finite and infinite number of conjugacy classes of finite subgroups of $\Diff^+(M)$  can be largely analyzed by the methods of this paper, but works out somewhat differently (e.g., one has finiteness in some cases of one dimensional singular set) and is especially  more involved when the singular set is 2-dimensional.  We shall not discuss this here.

Finally, let us present our result from an additional point of view: Given $H$ as above and $\Gamma$ a uniform lattice in it.  It acts via the standard action $\rho_0$ by translation on the symmetric space $H/K$ which topologically is $\bbr^d$. The Farrell-Jones topological rigidity result implies that
if $\Ga$ is torsion free,
every proper discontinuous (orientation preserving) action $\rho$ of $\Ga$ on $H/K$ is conjugate within $\Homeo^+ (H/K)$ to $\rho_0$.  It has been known for a long time (cf. \cite{We2} for discussion and references) that this is not necessarily the case if $\Gamma$ has torsion.  Our discussion above (with some additional ingredient based on \cite{CDK}, \cite{We2} - see \S 7)   gives the essentially complete picture.  But first a definition:

\begin{definition*}  For $\Ga, H, K$ and $\rho_0$ as above, say
\begin{enumerate}
\item The lattice $\Ga$ has topological strong rigidity if every proper discontinuous action $\rho$ of $\Ga$ on $H/K$, is conjugate to $\rho_0$ by an element of $\Homeo^+(H/K)$.
    \item $\Ga$ has local topological rigidity if for every proper discontinuous action $\rho$ of $\Ga$ on $H/K$,  there exists a small neighborhood $U$ of $\rho$ in $\Hom (\Ga, \Homeo^+ (H/K)$) such that any $\rho' \in U$ is conjugate to $\rho$ by an element of $\Homeo^+ (H/K)$).
        \end{enumerate}
        \end{definition*}
        The following two results follow from Corollary 1.6 and Theorem 1.5 (see \S 7):

        \begin{thm}\label{1.7} Let $H$ be a semisimple group, $K$ a maximal compact  subgroup and $\Ga$ an irreducible uniform lattice in $H$.  Then $\Ga$ satisfies the topological local rigidity if and only if for every non-trivial element of $\Ga$ of finite order, the fixed point set of its action on $H/K$ is zero dimensional.
        \end{thm}
\begin{thm}\label{1.8}
For $H, K$ and $\Ga$ as in Theorem 1.7 but assuming \\ $\dim (H/K) \neq 2, 4$.  Then one of the following holds:

\begin{enumerate}[{\rm(a)}]
 \item $\Ga$ has topological strong rigidity, i.e., it has a unique (up to conjugation) proper discontinuous action on $H/K \simeq \bbr^n$.
    \item $\Ga$ has an infinite countable number of such actions, yet all are locally rigid.
    \item $\Ga$ has uncountably many (conjugacy classes) of such actions.
    \end{enumerate}
    These cases happen if and only if the following hold, respectively:
    \begin{enumerate}[{\rm(a)}]
    \item
    \begin{enumerate}[{\rm (i)}]
    \item $\Ga$ acts freely on $H/K$ or
    \item  every torsion (i.e.,  non-trivial of finite order) element of $\Ga$ has 0-dimensional fixed point set in $H/K$ and either $\dim(M) \equiv 0 \\ (mod \,  4)$ or there are no elements of order 2.
        \end{enumerate}
        \item $\dim(H/K) \equiv 2 (mod \,  4)$, the fixed point set of every torsion element is 0-dimensional and there is some element of order 2, or
            \item  there exist a torsion element in $\Ga$ with a positive dimensional fixed point set.
    \end{enumerate}

\end{thm}

The paper is organized as follows. In \S 2, we prove Theorem 1.2. In \S 3, we give preliminaries for the proof of Theorem 1.5, which will be given in \S 4.   In this proof we depend crucially on the deep works of Farrell and Jones \cite{FJ1} \cite{FJ2} and Bartels and Lueck  \cite{ BL} related to the (famous) Borel conjecture as well as work  of \cite{CDK}.  In \S 5, we analyze manifolds of dimension 4, while in \S 6 we  prove Theorem 1.3.  Section 7 discusses topological rigidity of lattices and proves Theorem 1.7 and 1.8.

\begin{remark*} If one allows orientation reversing actions then if $\dim M\ge 7$ there is a trichotomy theorem; rigidity holds if the action is free or the $\dim M \equiv 1 \mod 4 $ or if $\dim M \equiv 3 (4)$ and the elements of order 2 act freely.  The proof of this is similar to the one we give below for Theorem 1.5.  We believe that the remaining cases, at least when $\dim M \neq 4$, work out similarly to that theorem.
\end{remark*}

\bigskip

\noindent {\bf Acknowledgment:} This work was partially done while the authors visited Yale University.  They are grateful to Dan Mostow for a useful conversation.
We are also grateful to NYU, the Hebrew University, ETH-ITS for their hospitality as well as to the NSF, ERC, ISF, Dr. Max R\"ossler, the Walter Haefner Foundation and the ETH Foundation, for their support.
\section{Proof of Theorem 1.2}

The proof of the theorem is in four steps:

\medskip

\noindent{\bf  Step I:} In [BeL],  M. Belolipetsky and the second named author showed that for every $n \ge 3$ and for every finite group $G$, there exist infinitely many closed, oriented,  hyperbolic manifolds $M = M^n(G)$ with ${\Isom}^+(M) \simeq G$.  More precisely, it is shown there that if $\Ga_0$ is the non-arithmetic cocompact lattice in $H = PO^+(n, 1)$ constructed in [GPS], then it has infinitely many finite index subgroups $\Ga$ with $N_H(\Ga)/\Ga\simeq G$.  The proof shows that $\Ga$ can be chosen so that $N_H(\Ga)$ is torsion free and moreover $N_H(\Gamma)/\Ga \simeq \Isom^+ (M) = \Isom(M)$ for $M = \Gamma \backslash \bbh^n$.  This implies that $G = N_H(\Ga)/\Ga$ acts on $M$ freely, a fact we will use in Step IV below.

Let $M = M^n(G) $ be  one of these manifolds, $\Ga = \pi_1 (M)$.  So $\Ga$ can be considered as a cocompact lattice in $\Isom^+ (\bbh^n) = PO^+(n, 1)$,  the group  of orientation preserving isometries of the $n$-dimensional hyperbolic space $\bbh^n$.

\medskip
\noindent{\bf  Step II:}  The Mostow Strong Rigidity Theorem [Mo] for compact hyperbolic manifolds  asserts that if $\Ga_1$ and $\Ga_2$ are torsion free cocompact lattices in $\Isom (\bbh^n)$, then every group theoretical isomorphism from $\Ga_1$ to $\Ga_2$ is realized by a conjugation within $\Isom(\bbh^n)$ (or in a geometric language, homotopical equivalence of hyperbolic manifolds implies an isomorphism as Riemannian manifolds).  Applying Mostow's theorem for the automorphisms of $\Ga = \pi_1 (M)$ implies that $\Aut(\Ga)$ can be identified with $N_{\Isom(\bbh^n)} (\Ga)$, the normalizer of $\Ga$ in $\Isom(\bbh^n)$.  Hence the  outer automorphism group $\Out (\Ga) = \Aut(\Ga)/\Inn(\Ga)$ of $\Ga$ is isomorphic to $N_{\Isom(\bbh^n)} (\Ga)/\Ga$ and hence also to $\Isom (M)$, which in our case is equal to $\Isom^+(M)$ by step I.

\medskip
\noindent{\bf  Step III:}  In [Bo], Borel showed that if $\Ga$ is a torsion free cocompact lattice in a simple non-compact Lie group $H$, with a maximal compact subgroup $K$ and associated symmetric space $X = H /K$, then every finite subgroup $F$ of $\Homeo(M)$ where $M = \Ga\backslash X$, is mapped injectively into $\Out (\pi_1 (M)) = \Out (\Ga)$ by the natural map.  If $F \le \Homeo^+(M)$, then its image is in $\Out^+(\Ga)$ which the kernel of the action of $\Out(\Ga)$ on $H^n(\Ga, \bbz)\simeq \bbz$, so $[\Out(\Ga):\Out^+ (\Ga)] \le 2$.
Borel's result is actually much more general; the reader is referred to that short paper for the general result and the proof which uses Smith theory and cohomological methods.

Anyway, applying Borel's result for our $M = M^n(G)$ finishes the proof  of the first part of Theorem 1.2.
In particular, one sees that in all these examples, $G = \Isom^+(M)$ acts freely on $M$ since the isometry group, in this case, is the group of covering transformations which act freely on $M$.

\medskip
\noindent{\bf Step IV:}
We have shown so far that whenever a finite group $F$ acts on $M$ as above, there is a natural injective homomorphism $F \hookrightarrow  \Out^+ (\pi_1 (M))\cong \Isom^+ (M)$.  Denote the image of $F$ in $\Isom^+(M)$ by $L$.  Our next goal is to show that $F$ is conjugate to $L$ within $\Homeo^+(M)$.
 For ease of reading we will call $M$ with the
action of $F$, $M'$, to avoid confusion.

There is actually an equivariant map $M'\toM$ that is a homotopy equivalence.
To see this,  note that $N_H(\Ga)$ is torsion free and hence so is ${\ol\Ga}$, the preimage of $L$ in $N_H(\Ga)$ w.r.t. the natural projection $N_H(\Ga)\to \Out(\Ga) = \Out(\pi_1(M))$. Similarly, let us consider
 all of the possible lifts  of all of the elements of $F$
to the universal cover, which form a group ${\ol \Ga}'$ (the orbifold fundamental group of $M'/F$, which we presently show  is the genuine fundamental group)
that fits in an exact sequence:
$$1\to\Ga (=\pi_1(M))\to{\ol \Ga}' \to F \to 1$$

 As $\Ga$ is centerless and $F$ and $L$ induce the same outer automorphism group,  it  follows  that ${\ol\Ga}'$ is also torsion free and as a corollary $F = {\ol\Ga}'/\Ga$ acts freely on $M'$.  Hence $M'/F$ is homotopy equivalent to $M/L$ as both have ${\ol \Ga}'\simeq {\ol\Ga}$ as their fundamental group.  By the Borel conjecture for hyperbolic closed manifolds (which is a Theorem of Farrell and Jones [FJ1] for $n\ge 5$ and of Gabai-Meyerhoff-Thurston [GMT] for $n=3$) the map $M'/F\to M/L$ is homotopic to a homeomorphism which preserves $\pi_1(M') = \pi_1(M)$, as did the original homotopy equivalence.
Since liftability in a covering space is a
homotopy condition, this homeomorphism can be lifted to the cover $ M'\to M$, producing a
conjugating homeomorphism between the actions. Theorem 1.2 is now proved.

In summary, the above proof is analogous to (and relies on)  the fact that Mostow rigidity gives a uniqueness of the isometric action (or in different terminology, the uniqueness of the Nielsen realization of a subgroup of $\Out(\Ga)$).  At the same time, the Farrell-Jones/Gabai-Meyerhoff-Thurston rigidity gives the uniqueness of the topological realization in the case of free actions.  We will see later that this freeness condition is essential.

\section{Some ingredients for the proof of Theorem 1.5}

The proof of Theorem 1.5 is based on  results, sometimes deep theorems, some of which are well-known and others  which might be folklore (or new).  We  present them in this section and use them in the next one.

\begin{ingr} For $v\ge 3$, there exists infinitely many non-simply connected homology spheres $\Sigma^v$, each bounding a contractible manifold $X^{v+1}$ such that the different fundamental groups $\pi_1(\Sigma)$ are all freely indecomposable and are non isomorphic to each other.  Moreover, $X \times [0,1]$ is a ball.
\end{ingr}
\begin{proof}  For $v > 4 $ this is very straightforward.  According to Kervaire \cite{Ker}, a group $\pi$ is the fundamental group of a $(PL)$ homology sphere iff it is finitely presented and superperfect (i.e.,  $H_1 (\pi) = H_2(\pi) = 0)$.   Moreover every $PL$ homology sphere bounds a $PL$ contractible manifold (this is true for $v \ge 4$, and for $v = 3$  in the topological category [Fr]). The product of a contractible manifold with $[0, 1]$ is a ball as an immediate application of the $h$-cobordism theorem (see [Mi1]).

 For $v = 3$, we could rely on the work of Mazur [Ma] in the $PL$ category, but would then need to use subsequent work on the structure of manifolds obtained by surgery on knots.  Instead,  as we will be working in the topological category, we  rely on [Fr] which shows that the analogue of all of the above holds topologically for $v = 3$, aside from the characterization of fundamental groups: however, using the uniqueness of the JSJ (\cite{JS}, \cite{J}) decomposition of Haken 3-manifolds, homology spheres obtained by gluing together nontrivial knot complements are trivially distinguished from one another.

 For $v = 4$, note that all the fundamental groups of the $v = 3$ case arise here as well: if $\Sigma^3$ is a homology sphere then $\partial(\Sigma^o \times D^2)$ is a homology 4-sphere with the same fundamental group (where $\Sigma^o$ denotes, as usual, the punctured manifold).
\end{proof}

We also need:
\begin{ingr} For $m-1 > c_0\ge 3$ and every orientation preserving linear free action $\rho$ of $G = \bbz_p$ on $S^{c_0}$ (in particular, $c_0$ is odd), there exist an infinite  number of homology spheres $\Sigma^{c_0}$ with non-isomorphic fundamental groups and with a $G = \bbz_p$-free action satisfying: For each such $\Sigma $ there exists an action of $G$ on $B^m$ fixing $0 \in B^m$ such that
\begin{enumerate}
\item{} The action of $G$ on $S^{m-1}$ is isomorphic to the linear action $\rho \oplus$ Identity, and
\item{} the local fundamental group $\pi^{\hbox{local}}_1 (B^m\setminus F, 0)$ is isomorphic to $\pi_1(\Sigma)$, when $F$ is the fixed point set. Moreover, this action is topologically conjugate to a PL action on a polyhedron.
\end{enumerate}
\end{ingr}

Let us recall what is meant by the local fundamental group: This is the inverse limit $\lim\limits_{\substack{\longleftarrow\\ \a}}$
$\, \pi_1(U_\a, x_\a)$ where the $\{ U_\a\}$ is a sequence of connected open neighborhoods converging down to $0$, and $x_\a \in U_\a\setminus F$ is a sequence of base points.  Note that by the Jordan Curve Theorem, $U_\a\setminus F$ is connected as codim $F\ge 2$.  Also the induced maps are well defined up to conjugacy, so the limit is well defined.

\begin{proof} For every homology sphere $\Sigma'$ of odd dimension $c_0$, let $\Sigma = p \Sigma' = \Sigma'\# \Sigma'\#\cdots\# \Sigma'$ $p$ times.  We now give $\Sigma$ a free $\bbz_p$ action, by  taking connected sum along an orbit of the free linear action on $S^{c_0}$ with the permutation action on $p\Sigma'$.  The action on $S^{c_0}$  bounds a linear disk $D^{c_0 + 1}$.  One can take the (equivariant) boundary connect sum of this disk with $pX, X$ the contractible manifold that $\Sigma'$ bounds to get a contractible manifold $Z$ with  $\bbz_p$ action fixing just one point which is locally smooth at that
point and has the given local representation $\rho$ there.

For motivation, consider now $c(\Sigma) \times B^{m - c_0 - 1}$ where $c(\Sigma)$ is the cone of $\Sigma$.  It is a ball, by Edwards's theorem \cite{D} (combined with the $h$-cobordism theorem), and has an obvious $\bbz_p$ action as desired except that the action on the boundary is not linear: the fixed set is $S^{m-c_0 - 2}$ but it is not locally flat.

Instead, let $Z$ be the locally linear contractible manifold bounded by $\Sigma$, constructed above.  The manifold $Z \cup (\Sigma \times [0, 1]) \cup Z$ is a sphere (by the Poincar{\'e} conjecture).  If one maps this to $[0, 1]$ by the projection on $\Sigma \times [0, 1] $ and extending by constant maps on the two copies of $Z$, then the mapping cylinder of this map $\varphi:Z \cup (\Sigma \times [0, 1]) \cup Z \to [0, 1]$ is a manifold, again by Edwards's theorem.  It has an obvious $\bbz_p$ action with fixed set an interval.  The action on the boundary sphere is locally smooth with two fixed points, so that an old argument of Stallings \cite{St2}  shows that it is topologically linear with $\rho $ as above.  Note that the nonlocally flat points of the fixed point set correspond to the points where the local structures is $ c(\Sigma) \times [0, 1]$; hence  the local fundamental group is $\pi_1(\Sigma)$, as required.  This proves the result for the case $m = c_0 + 2$.

For $m > c_0+2$, one can spin this picture: Map $(S^{m-c_0 - 1 } \times Z^{c_0 + 1}) \cup (B^{m-c_0 - 1} \times \Sigma)$ to $B^{m-c_0}$ in the obvious way and again the mapping cylinder produces a ball with locally linear boundaries and desired fixed set. This time the linearity of the boundary action follows from Illman \cite{I}.
\end{proof}

We will also need the following group theoretical result:

\begin{prop} If $\{ \pi_i\}^\infty_{i = 1} $ and $\{ \pi'_i\}^\infty_{i = 1}$ are two infinite countable families of non-isomorphic freely  indecomposable groups such that
$\mathop{\ast}\limits^\infty_{i = 1} \pi_i$  is isomorphic to $\mathop{\ast}\limits^\infty_{i = 1} \pi'_i$, then after reordering for every $i$, $\pi_i$ is isomorphic to $\pi'_i$.
\end{prop}
\begin{proof}
Recall that by the Bass-Serre theory, a group $\Gamma$ is a free product $\mathop{*}\limits^\infty_{i = 1} \pi_i$ if and only if $\Gamma$ acts on a tree $T$ with trivial edge stabilizers and a contractible quotient  and with one to one correspondence  between the vertices of $T$ and the conjugates of $\pi_i (i \in \bbn)$ in $\Gamma$, where each vertex corresponds to its stabilizer.  Now assume $\Gamma \simeq \mathop{*}\limits^\infty_{i = 1} \pi_i$  and also $\Gamma \simeq \mathop{*}\limits^\infty_{j = 1} \pi'_j$ with the corresponding trees $T$ and $T'$.  Fix $i \in\bbn$, as $\Gamma$ acts on $T'$ with trivial edge stabilizers and $\pi_i$ is freely indecomposable, $\pi_i$ fixes a vertex of $T'$.  Hence there exists $j \in \bbn$ s.t. $\pi_i \subseteq \pi_j^{'\tau}$.  In the same way $\pi^{'\tau}_j$ is a subgroup of some $\pi^\delta_k$ for some $\delta \in \Gamma$.  This means that $\pi_i \subseteq \pi^\delta_k$.  But in a free product  a free factor cannot have a non-trivial intersection with another factor or with a conjugate of it.  Moreover, if $\pi_i \cap \pi^\delta_i \neq \{ e \}$, then $\pi^\delta_i = \pi_i$.
Indeed, if $g$ is in this intersection,  it fixes the fixed vertex of $\pi_i$ as well as that of $\pi^\delta_i$,  hence also the geodesic  between them, in contradiction to the fact that $\Gamma$ acts with trivial edge stabilizers.

We deduce that $\pi_i \subseteq \pi_j^{'\tau} \subseteq \pi_i$ and hence $\pi_i = \pi_j^{'\tau}$.

This shows by symmetry that  the collections $\{ \pi_i\}$ and $\{ \pi'_j\}$ are identical.
\end{proof}

Finally, let us recall the famous Borel conjecture which asserts that two aspherical manifolds which are homotopy equivalent are homeomorphic.  Moreover, the original homotopy equivalence is homotopic to a  homeomorphism.  This conjecture of Borel was proved by Farrell and Jones \cite{FJ2} for  the locally symmetric manifolds $M$ discussed in this paper, if $\dim (M) \ge 5$ and by Gabai-Meyerhoff-Thurston \cite{GMT} for the case of $\dim (M) = 3$.  This, in particular,   says that there is a unique cocompact proper topological action of $\Ga = \pi_1(M)$ on the symmetric space $H/K$ for any uniform torsion-free lattice.

But if ${\ol \Ga} \tle \Ga$ is a finite extension with torsion, then the situation is more delicate. In fact, as we will see, there is no rigidity anymore in the topological setting and ${\ol \Ga}$ may have many inequivalent actions on $H/K$.  In other words, the ``equivariant Borel conjecture" is not true.  It is of interest (though not really relevant to the goals of this paper) to compare this with the analogous situation in the setting of $C^*$-algebras, where the analogue of the Borel conjecture is the Baum-Connes conjecture.  This latter conjecture is known to be true in many situations, even in its equivariant form, i.e., for groups with torsion,  while the equivariant Borel conjecture fails in some of those cases (\cite{CK}, \cite{Q}, \cite{We1}, \cite{We2}).   The failure is due to the non-vanishing of the Nil and the first author's UNil groups (see \cite{BHS}  and \cite{Ca})\footnote{More accurately, these algebraic reasons explain the failures of the equivariant Borel conjecture relevant here.  \cite{We2} gives other sorts of examples when the singular set is high dimensional.}.  The latter is the source of non-rigidity in the case of isolated singularities.

The specific outcome relevant for our needs is the following:

\begin{ingr}
If $\triangle$ is a lattice containing $\Ga$ a torsion free uniform lattice in $H$  as a normal subgroup of finite index, so that the normalizer of any finite subgroup is finite, then the proper discontinuous actions of $\triangle $ on (the topological manifold) $H/K$ are in a one to one correspondence with the action of $\triangle/\Ga$ on $\Ga\setminus H/K$, inducing the given outer automorphism of $\Ga$.  If $\dim (H/K ) \neq 4$ then this action is unique unless $\dim (H/K)$ is 2 mod 4 and $\triangle$ contains an infinite dihedral subgroup. In that case, the number of conjugacy classes is infinite and countable.
\end{ingr}

\begin{proof} Any proper action of $\triangle$ is automatically free when restricted to $\Ga$ (by torsion freeness).  The action is cocompact, because if it were not, this quotient space would show that the cohomological dimension of $\Ga$ is less than $\dim(H/K)$ which cannot happen, since there is a cocompact  action.  As a result, the Borel conjecture, proved for uniform lattices by \cite{FJ2}, shows that all of these actions are standard for the $\Ga$ subgroup, i.e.,  equivalent to the original action of $\Ga$ on $H/K$.

Note that this argument did not use the fact that the manifold on which $\triangle$ acts is $H/K$;  it would apply automatically to any contractible manifold.  This shows that such a manifold is automatically Euclidean space, as follows quite directly from \cite{St1}.

With this preparation, the result now follows from [CDK] together with \cite{BL}.  The condition on normalizers is equivalent to the discreteness of the fixed point set.  (For every finite group $G$, $N_H (G)/G$ acts properly on the fixed set of $G$ on $H/K$.)  Assuming the Farrell-Jones conjecture for $\triangle$, which is a theorem of \cite{BL}, \cite{CDK} gives a description of the set of actions in terms of UNil groups and maximal infinite dihedral subgroups $\triangle$.\footnote{Essentially the argument shows that, unlike what is done in the next section in the situation where the singular set is positive dimensional, the action of $\triangle/\Ga$ is equivariantly homotopy equivalent to the linear one.  Since the singular set is very low dimensional, one can promote this to an isovariant homotopy equivalence.  At that point, surgery theory can be used to reduce this problem to issues in $K$-theory and $L$-theory that are handled by the Farrell-Jones conjecture.  It turns out that the $K$-groups of $\triangle$ are the limit of those of the finite subgroups of $\triangle$; however, because of UNil, the analogous statement is not true for $L(\triangle)$ and the calculation, {\sl in this case}, reduces to the infinite dihedral subgroups.} A lattice that contains an infinite dihedral subgroup, contains a maximal one (by discreteness: the $\bold Z$ subgroups cannot keep growing in a nested sequence, since they correspond to shorter and shorter closed geodesics and a compact manifold has a positive injectivity radius).

\end{proof}

\section{Proof of theorem 1.5}

	For Theorem 1.2, we have depended on the fact that the construction of [BeL] produced free actions.  The constructions we presently describe show that whenever a manifold $M$ of  dimension $\ge 5$ has  an (orientation preserving) action whose singular set (i.e.,  the union of the fixed sets of all nontrivial subgroups) is positive dimensional, there are actually continuously  many actions on $M$ that induce the same outer automorphisms of their fundamental group but are not topologically conjugate.

To prove part (c) of the Theorem it suffices to prove it for cyclic group, i.e., that if a cyclic group $G = \bbz_p,\; p $ prime, has a positive dimensional fixed point set, then there is a continuum  of such non-equivalent actions.

	  Let $V$ be a component of the fixed set of the action of $G$.  Let $v = \dim(V)$ and let $\rho$ be the normal representation of $\bbz_p$ on $\bbr^c$  ($c = m-v$, the codimension of $V$ in $M$).  If $v > 2$, take $\Sigma^v$ and $X^{v + 1}$ as in Ingredient 3.1.  The product $X\times D(\rho)$ is a ball  with an action of $\Z_p$ whose fixed set is $X$.  The action of $G$ on the boundary $\partial(X\times D^c(\rho)) = (X \times S^{c-1}(\rho)) \cup_{ (\Sigma\times S^{c-1}(\rho))} (\Sigma \times D(\rho)$ has $\Sigma $ as its fixed set (as $G$ has no fixed points on $S^{c-1}(\rho)$ and a unique fixed point - the origin  - in the disk $D(\rho))$.  The normal representation to this fixed set is still $\rho$.   We can take equivariant connected sum of $M$ with this $G$-sphere to obtain a new $G$-action on $M \# S^{c + v} \simeq M$ whose fixed set is $V\# \Sigma$.  Since the fixed set of the new action  does not have the same fundamental group as $V$ (e.g., by Grushko's theorem as $\pi_1(\Sigma)$  was assumed nontrivial), it is not conjugate to the original action. Of course  it induces the same outer automorphism on $\pi_1$. Notice that we can think of this procedure as being a local equivariant insertion; near a point $x \in V$  we modify the action of $G$ only in a small specified ball.  This procedure can be done any finite number of times to get countably  many non-conjugate actions.

In fact, we can even get a continuum of actions of $G$ on $M$. Let us first make a definition: For a finite group $G$ acting topologically on a manifold $M$, we say that $x$ in $M$ is a  {\sl decent fixed point}, if the action of $G$ in some open neighborhood of $x$ is topologically conjugate to a simplicial action on a polyhedron.
 Now, apply the process above with smaller and smaller disjoint balls  in $M$ converging to some point $x_0$ using any choice of $\pi_1(\Sigma)$'s provided by Ingredient 3.1.  The outcome is a copy of $M$ with an action of $G$ on it with a fixed point set $W$ containing $x_0$, which is the unique non-decent point on $M$. This set $W$ is not a manifold, but $W \setminus \{ x_0\}$ is.   Moreover, the fundamental group of $W\setminus \{ x_0\}$  is isomorphic to the free product of the fundamental group of the  original set $V$ with the free product of the infinitely many different $\pi_1(\Sigma)$'s which have been used.   Now, if two such constructions lead to equivalent actions of $G$ on $M$, then this unique non-decent point of the fixed point set should be preserved.  Proposition  3.3  would imply that the two collections of $\pi_1 (\Sigma)$'s are equal.
An infinite countable set has a continuum number of subsets and  we can therefore get a continuum number of non-conjugate actions of $G$.  This proves case (c) for $v \ge 3$.

In the case of $v \ge 3$, we replaced balls $B^m$ (with isometric local $G$-action) of $M$, by copies of the same ball with new $G$-actions.  The new action preserved the original normal action $\rho$ but changed dramatically the fixed point set.  For $v = 1 $ or $2$, we are not able to argue like that (for  lack of suitable $\Sigma$'s as above).  Instead we will keep the fixed point set in $B^m$ but deform the normal action $\rho$.  In fact,  this second method works whenever $c = m - v > 2$, so altogether the two methods cover all cases if $\dim M \ge 5$.

One  now imitates the procedure described before to modify the original isometric action of $G$ on $M$ at a ball around a fixed point by replacing it with some $B^m$  as in Ingredient 3.2.  The resulting action is not equivalent to the original one as $\pi_1(\Sigma)$ can be recovered from it as the local fundamental group at a fixed point.  Doing this procedure any finite number of times with different $\Sigma$'s each time, gives us an infinite countable collection of non conjugate actions.  To get a continuum, we argue as before.  In fact, this version is easier: The family  of $\pi_1(\Sigma)$'s used can be recovered from the action of $G$ on $M$ as being exactly the collection of non-trivial local fundamental groups at decent fixed points of $G$.

This finishes the proof of part c of Theorem 1.5.

Now, the proof of part (a) is exactly the same as Step IV in the proof of Theorem 1.2.

Part (b) can be deduced from Ingredient 3.4.  The existence of a dihedral subgroup in  ${\ol\Ga}$ is equivalent to the existence of an involution in $G$ fixing a point in $M$.  This is indeed the case: If ${\ol \Ga}$ contains a dihedral group, it contains an element of order 2.  This element has a fixed point on $H/K$ and hence also on $M$.  In the other direction: assume $\tau \in G$ is an involution fixing a point $p$ of $M$.  Then by \cite{CF}, $\tau$ has at least a second fixed point $q$. Let $\a$ be a geodesic from $p$ to $q$, then $\tau(\a)$ is another such geodesic and indeed, $\a \cup \tau (\a)$ is a closed geodesic $\gamma \in \pi_1 (M, p) = \Ga$.  The group generated by $\gamma$ and $\tau$ is a dihedral group.

This finishes the proof of Theorem 1.5 for $\dim M \ge 5$.

To prove Theorem 1.5 for dimension 3 observe that case (b) does not occur; since the fixed point set is of even codimension. Now, case (a) is exactly as before, with this time the work of Gabai-Meyerhoff-Thurston [GMT] replacing the work of Farrell and Jones.

For part (c), we use the same procedure of replacing a ball around one fixed point by an exotic action. This time the work of Bing [Bi], provides us with uncountably many actions of $G = \bbz_p$ on $B^3$.  (Bing discusses $\bbr^3$, but his construction clearly works on $B^3$ and produces actions with given linear $G$ action on $\partial B^3$.) As the actions are distinguished by the structures of this non-locally flat set as a subset of the line, they remain inequivalent in any manifold.
$\square$

\section{The case of $\dim M = 4$}

Theorem 1.5 may hold true in dimension 4 as well, but we can only prove  part (c) of it, namely:
\begin{thm} Let $M$ be a locally symmetric irreducible manifold of dimension 4, with a $G = \bbz_p$ faithful isometric action whose fixed point set $V$ is positive dimensional.  Then $G$ has a continuum number of inequivalent topological actions on $M$.
\end{thm}
\begin{proof} First note that $\dim V = 2$ as the codimension must be even.  The Smith conjecture asserts that if $\bbz_p$ acts topologically on $S^n$, $n \ge 3 $ with a smooth $S^{n-2}$ as its fixed point set then $S^{n-2}$ is isotopic to the unknot.  While this conjecture is true for $n = 3$, it  turns out to be false when $n \ge 4$  as we now discuss.
\begin{thm} There exist an infinite number of $G = \bbz_p$ actions $\rho_i, \, i \in \bbn$  on $S^4$ satisfying:
\end{thm}
\begin{enumerate}
\item{} \emph{The fixed point set is a non trivial knotted $S^2 \subset S^4$.}
\item{} \emph{The complement of the knot is a 4-dimensional manifold $N_i$,  which fibers over the circle, with fibres $F^0_i$, each a 3-dimensional manifold with boundary $S^2$.  Once this $S^2$ is filled by a ball,  the resulting closed manifold $F_i$  is irreducible.}
    \item{} \emph{For every $i\! \neq \! j,\;  \pi_1 (F_i)$ is not isomorphic to a subgroup of $\pi_1 (F_j)$.}
    \end{enumerate}

    \smallskip

    Note that $\pi_1 (F_i) = \pi_i (F^0_i)$ and this is a normal subgroup of $\pi_1(N_i)$, equal to $[\pi_1(N_i), \pi_1(N_i)]$ with the quotient isomorphic to $\bbz$.  Also, as $F_i$ is irreducible, $\pi_1 (F_i)$ is freely indecomposable.  We postponed the proof of Theorem 5.2, using it first to prove 5.1.

    Given $x \in V\subset M$, we will replace a punctured sphere $S^4$ around $x \in M$  (with its $G$-action) by a punctured $S^4$ around a fixed point provided by Theorem 5.2.  (We can adjust the action  on the sphere to have the same normal representation as that of $V$ just by changing the generator of $G = \bbz_p$.)  More precisely, as in the proof of Theorem 1.5, we replace $M$ by $M \# S^4$, but the latter is homeomorphic to $M$.  The new fixed point set is $V\# S^2$.

    This time, neither the fundamental group of the fixed set nor the local fundamental group distinguishes the two actions.  Thus we will argue slightly differently:  Look at the universal cover $\tilde M$ of $M$.  Take $\tilde V$ to be one of the  components of the lift of $V$ in $\tilde X$ containing a lift of $x$ (actually $\tilde V$ is the universal cover of $V$ since $\pi_1 (V)$ injects into $\pi_1(M))$. The complement of it has fundamental group isomorphic to $\bbz$ (as by Hadamard's Theorem this is diffeomorphic to the linear inclusion $\tilde V \simeq \bbr^2 \subset \bbr^4 \simeq \tilde M$, whose complement is homotopy equivalent  to a circle).  On the other hand, if $\tilde W$ is a lift of $W$ (= the fixed points of the modified action) then the fundamental group of its complement is an infinite amalgamated free product of $\pi_1(N_i)$ with itself amalgamated along $\bbz $ (generated by the meridian).  This group is  certainly not isomorphic to $\bbz$ and hence the two actions are not equivalent.  Repeating this with the different $N_i$'s gives us countably many inequivalent actions, as the different $\pi_1 (N_i)$'s are not isomorphic to each other. (These groups have $\pi_1 F_i * \pi_1 F_i * \dots $ as their commutator subgroups.)

    To get a continuum number, we argue as follows:  Let $\Omega$ be an infinite subset of $\bbn$.  For every fixed $i \in \Omega$, apply the procedure above infinitely many times around disjoint balls in $M$ converging to a point $x_i \in M$.  Do it in such a way that the set $\{ x_i\}_{i \in \Omega}$ has a unique limit point $x_0$.

    We got therefore a new action on $G$ on $M$ depending on $\Omega$.  We want to show that different $\Omega$'s lead to non-equivalent actions.  Indeed, the indecent points are $\{ x_i\}_{i \in\Omega}$ as well as $x_0$ (which is a unique limit point of the indecent points). Now if two families $\Omega$ and $\Omega'$ lead to equivalent actions then the conjugating   homeomorphism takes, after reordering, $x_i, i \in \Omega$ to $x_j', j' \in \Omega'$.  Looking, as before,  at the neighborhood of a lift of $x_i$ (and $x_{j'}$,)  in the universal cover, we deduce that $\pi_1 (F_{j'})$ is a subgroup of the infinite amalgamated product of infinitely many copies of $\pi_1(N_i)$. In fact it is in the (unique) kernel of the map from this group onto $\bbz$.  Thus $\pi_1 (F_{j'})$ is a subgroup of this kernel which is just a free product of infinitely many copies of $\pi_1 (F_i)$.  By the Kurosh Subgroup Theorem every subgroup of this kernel is a free product of subgroups of $\pi_1 (F_i)$ and of a free group.  As $\pi_1 (F_{j'})$ is freely indecomposable  (since $F_{j'}$ is irreducible) we deduce that $\pi_1 (F_{j'})$ is isomorphic to a subgroup of $\pi_1 (F_i)$.  By part (3) of Theorem 5.2 this implies $i = j'$ and hence $\Omega = \Omega'$ and Theorem 5.1 is now proven.
    \end{proof}

Let us now prove Theorem 5.2.

\begin{proof}  There are many ways of doing this construction.  We rely essentially on the work of \cite{Gi} and \cite{Z} and inessentially on work of \cite{Mi2} to choose an explicit method.  In \cite{Z}, Zeeman modified Artin's spinning construction of knots in $S^4$ to twist spinning: for a knot $K$ in $S^3$ the $q$-twist spin of $K$ is a knot in $S^4$ whose complement fibers over the circle, and whose fiber is the $q$-fold branched cover of $S^3$ (punctured) branched over $K$.  Its monodromy is exactly the deck transformation.  If $p$ is prime to $q$, the monodromy has a $p$-th root, which can be used to build a $\bbz_p$
action on the complement.  Using the 4-dimensional Poincar{\'e} conjecture [Fr], Giffen observes\footnote{
Giffen's paper also shows that for $p$ odd, one can avoid using the Poincar\'{e} conjecture (which was not known in any category at the time that paper was written).} that this action can be extended to one on the sphere with the twist spun knot as fixed set.

If one starts with $(r, s)$ torus knots, one obtains the Brieskorn manifold associated to $(q, r, s)$ as the closed fibers.  These groups, as observed in \cite{Mi2} are central extensions of the $(q, r, s)$-triangle groups.  Their quotients by their centers have as torsion exactly the cyclic groups of order $\{ q, r, s\}$; then by, for example,  letting $q,r,s$ run through primes these groups do not embed in one another.
\end{proof}


\section{The case of $\dim M = 2$}

The phenomena in dimension 2  genuinely differ from those  in higher dimension.
Let us note that Step I and Step III in the proof of Theorem 1.2 work equally well in dimension 2.  In Borel's result there is no assumption on the dimension.  The paper \cite{BeL} assumes $n\ge 3$, but the result is true also for $n=2$.  In fact, it was proved earlier by Greenberg [Gr] with the following elegant argument:  He showed that in the Teichm\"uller space classifying the hyperbolic structures on a given surface $S_g$ of genus $g$, or better yet, classifying conjugacy classes of cocompact lattices of $PSL_2(\bbr)$ isomorphic to $\pi_1(S_g)$, almost every such lattice is maximal and non-arithmetic.  Hence by Margulis' criterion for arithmeticity, it is equal to its commensurability group. Now given a finite group $G$, choose $g$ large enough so that $\pi_1(S_g)$ is mapped onto $G$ (this is possible since $\pi_1(S_g)$ is mapped onto $F_g$, the free group on $g$ generators, so taking $g$ greater or  equal to the number of generators of $G$ will do).  Now, let $\Ga$ be a non-arithmetic maximal lattice with epimorphism $\pi : \Ga \twoheadrightarrow G$ with kernel $\land$.  Then $N_{PSL_2(\bbr)}(\land) = \Ga$ and so $\Isom^+(\Sigma) = G$ for $\Sigma = \land \backslash \bbh^2$,  as needed.

Moreover, it is also true that given a surface $S$ of genus $\sigma \ge 2$, there are only finitely many conjugacy classes of finite subgroups in $\Homeo^+(S)$. (These are, by the Nielsen realization theorem \cite{Ke}, in one to one correspondence with the conjugacy classes of finite subgroups of the mapping class group $MCG(S)$.)  The analysis below shows that this number is more than 1 for every genus, in contrast to parts (a) and (b) of Theorem 1.5, in spite of the fact that the singular set is always 0-dimensional.  To see the finiteness  note first that for
any finite group $G$ there are only finitely many  conjugacy classes of subgroups of
$G$. That all topological actions of finite groups on surfaces can be smoothed is classical \cite{K}. Smooth actions can be made
isometric on some hyperbolic structure either by direct construction (cut paste PL
methods) or by using the uniformization theorem: there is a unique hyperbolic structure
conformal to any invariant Riemannian metric, and that hyperbolic metric has an
isometric action of $G$.

The finiteness of the number of actions is either obvious by thinking of the data
required to reconstruct $\Sigma \to \Sigma/G$ in terms of the quotient manifold, ramification points,
and group homomorphisms from  $\pi_1$(Nonsingular part of $\Sigma/G$) $\to G$. (See our
discussion of the Riemann-Hurwitz formula below.) As one varies over all finite groups,
one has only a finite amount of data for any fixed genus.

Despite all this, let us show that Theorem 1.2 fails in dimension 2 in the strongest possible way, namely:

\begin{thm} For no finite group $G$ does there exists a $G$-weakly exclusive 2-dimensional closed manifold. In fact, for every genus $\sigma > 1$, the set of isomorphism classes of  finite maximal subgroups of $\Homeo^+(S_{\sigma})$ is finite with more than one element, while for $\sigma = 0 $ or $1$, there are no maximal finite subgroups.
  \end{thm}

  Recall first that closed, oriented  surfaces are classified by their genus $0\le \sigma\in\bbz$.  Clearly the surfaces of genus $0$ (the sphere) and genus $1$ (the torus) cannot be $G$-weakly exclusive  for any $G$ since each of them has self-automorphisms of unbounded finite order.  So from now on assume $\sigma\ge 2$.  Now, if $S = S_\sigma$ and $G$ is a finite group of automorphisms, then by \cite{K} and the
  Hurwitz bound  $|G|\le 84(\sigma-1)$. The   Riemann-Hurwitz formula asserts that in this case, letting  $\bar S = S / G, \,   \pi:S \to \bar S$ the quotient map which  is a ramified covering, ramified at $0 \le r \in \bbz$ points with ramification indices $m_1,\dots, m_r$ and if $\bar S$ is of genus $\rho$ then the following holds
\beq\label{eq.1}2\sigma - 2 = |G| (2\rho - 2 + \sum\limits^r_{i = 1} (1 - \frac{1}{m_i})).\ee

What is even more important for us is the converse.  Namely if $G$ is a finite group generated by elements $a_1, \dots, a_\rho, b_1, \dots, b_\rho, c_1, \dots, c_r$
where
\beq\label{eq.2} \prod\limits^\rho_{j = 1} [a_j, b_j] \prod\limits^r_{j = 1} c_i = 1 \; \; \text{and}\ee
\beq\label{eq.3} \text{for\ }\, i = 1, \dots, r, \; \; c_i \;  \text{is\ of\ order\ } m_i\ee
and if \eqref{eq.1} holds, then $G$ acts faithfully on $S= S_\sigma$ with quotient $\bar S = S/G$ of genus $\rho$ and ramification indices $m_1,\dots, m_r$.

The Hurwitz upper bound actually follows from equation \eqref{eq.1}: To get the largest $G$, for a fixed $\sigma$, one wants the  term in brackets on the right-hand side to be minimal but positive.  A careful elementary analysis shows that the smallest value is $\frac{1}{42}$  and it is obtained only if $\rho = 0, r = 3$ and $\{m_1, m_2, m_3\} = \{2, 3, 7\}$.
One also observes that if this value is not attained then the next one is $\frac{1}{24}$ with $\rho = 0, r = 3$ and $\{ m_1, m_2, m_3\} = \{ 1, 2, 8\}$. It is known that for infinitely many $g$'s, the upper  bound of $84(\sigma - 1)$ is attained but for infinitely many others it is not (cf. \cite{L}).  In the second case it follows that $|G| \le 48(\sigma - 1)$.

The converse result allows one to prove that various groups act faithfully on $S_\sigma$.  For example, by taking $\rho = 2$ and $r = 0$ we see that

\smallskip

\noindent(a) \ The cyclic group $c_{\sigma - 1}$ of order $\sigma - 1$ acts faithfully on $S_\sigma$.

\medskip

Similarly, by taking $\rho = 1, r = 2$  and $m_1 = m_2 = \sigma$

\smallskip

\noindent (b) \ The cyclic group $C_\sigma$ of order $\sigma$ acts faithfully on $S_\sigma$.

\medskip

Finally, the following group
$$H_\sigma = \langle x, y | x^4 = y^{2(\sigma + 1)} = (xy)^2 = x^{-1} y)^2 = 1 \rangle $$
is shown in [Ac] and \cite{Mc} to be of order $8(\sigma + 1)$.
(Note that by the two last relations  every element of $H_\sigma$ can be written as $x^ay^b$ with $0\le a < 4$ and $0 \le b < 2 (\sigma +1)$).  By taking as 3 generators $c_1 =x, c_2 = y$ and $c_3 = (xy)^{-1}$ which are of orders $4, 2(\sigma +1)$ and 2, respectively, one sees that equation \eqref{eq.1} is satisfied with $\rho = 0, r =3$ and $(m_1, m_2, m_3) = (4, 2(\sigma +1), 2)$ and hence:

\smallskip

\noindent (c) \ The group $H_\sigma$ of order $8(\sigma +1)$ acts faithfully on $S_\sigma$.

\medskip

Let us mention in passing that Accola \cite{A} and Macmillan \cite{Mc}  used (c) to prove a lower bound (as an analogue to the upper bound of Hurwitz) and they showed that for infinitely  many $\sigma$'s, this lower bound of $8(\sigma + 1)$ is best possible.

Back to our goal:  We want to show that $S = S_\sigma$ cannot be $G$-weakly exclusive for any finite group $G$.  Assume it is, then (a), (b) and (c) imply that $C_{\sigma-1}, C_\sigma$ and $H_\sigma$ are subgroups of $G$ and hence:

\beq\label{eq.4} \ell.c.m(\sigma-1, \sigma, 8(\sigma+1)) | |    G|.\ee
Now clearly $\ell.c.m (\sigma-1, \sigma, 8(\sigma + 1)) \ge \frac{(\sigma-1)\sigma(\sigma+1)}{2}$ and by the Hurwitz Theorem $|G| \le 84(\sigma - 1)$.  This implies $\frac 12 \sigma
(\sigma + 1) \le 84$, i.e.,  $\sigma \le 12$.

Now checking case by case for  $\sigma = 6,7,9,10,11,12$, one sees that $\ell.c.m (\sigma -1, \sigma, 8(\sigma +1)) > 84 (\sigma-1)$ in all these cases,  which leads to a contradiction. We are left with $\sigma = 8$ and $2\le \sigma \le 5$.

For $\sigma = 8$, we observe that if $S_8$ is $G$-weakly exclusive than by (b), $C_8$ is a subgroup of $G$ and so is $H_8$ of (c).  But $C_8$ is cyclic, while the 2-sylow subgroup of $H_8$ contains the {\it non-cyclic} subgroup of order 4, generated by $xy $ and $x^{-1}y$.  Thus the 2-sylow subgroup of $G$ is non cyclic and hence of order greater than 8, i.e.,  at least 16.  This implies that $G$ is of order at least $7\cdot 16\cdot 9 > 84\cdot 7$.  This contradicts  the Hurwitz upper bound and hence $S_8$ cannot be $G$-weakly exclusive.

To handle the case $\sigma = 5$, let us observe that the Hurwitz upper bound of $84(\sigma-1) = 336$ is obtained in this case.  Indeed,  look at $G=SL_2(7)$, a group of order 336 with the generators

$$ c_1 = \begin{pmatrix} 1 &1\\0 &1\end{pmatrix}, \; \; c_2 = \begin{pmatrix} 0 &1\\ -1 &0\end{pmatrix}$$ and $c_3 $ where
$$c^{-1}_3 = c_1 c_2 = \begin{pmatrix} -1 &1\\ -1 &0\end{pmatrix} $$
of orders $7, 2$ and $3$,  respectively.

Thus if $S_5 $ is $G$-weakly exclusive, $G$ must be $SL_2(7) $ since Hurwitz bound is attained for this group.  But $SL_2(7)$ does not contain the group $H_5$ of $(c)$ of order 48, since 48 does not divide 336.

Consider now the case $\sigma= 4$.  We claim that $Sym(5)$ of order 120 acts on it.  Indeed, taking $\rho = 0, r = 3$ and
$$ c_1 = (1,2,3,4,5), \; \; c_2 = (1, 2) \; \; \text{and \ } c_3 = (c_1c_2)^{-1} = (5, 4, 3, 1)$$
of orders 5, 2 and 4 respectively. We get a solution to (1)
and hence $Sym(5)$ acts on $S_4$.
Assume now that $S_4$ is $G$-weakly exclusive.
The Hurwitz bound $84.3 = 254$ cannot be obtained in such a case, since $120$ does not divide $254$.  Thus $G$ is of order at most $48\cdot 3 = 154$, but we know that its order should be divisible by 120.  Hence $|G| = 120$ and $G = \Sym(5)$.  But (c) above shows that $G$ should also contain $H_4$ which is of order 40.  As $\Sym (5) $ has no subgroup of order 40, we get a contradiction.   Hence $S_4$ is not $G$-weakly exclusive.

For $\sigma = 2$ and $3$, a full classification of the finite groups acting on $S_\sigma$ is given in \cite{Bro}.  From the list there it is clear that $S_\sigma$ is not $G$-weakly exclusive  also in these last two cases.  The Theorem is now fully proved.
\hfill $\square$

\smallskip

The Theorem says in particular that for every $g$, $\Homeo^+ (\Sigma_g)$ has at least two conjugacy classes of maximal finite subgroups (even isomorphism classes).  In fact, the number of those is unbounded as a function of $g$:

\begin{prop} The number of isomorphism classes of maximal finite subgroups of $\Homeo^+ (\Sigma_g)$ (or equivalently of $MCG(\Sigma_g)$) is unbounded as a function of $g$.
\end{prop}
\begin{proof} Let $\Ga$ be a fixed cocompact lattice in $PSL(2, R)$ which is the fundamental group of a surface of genus 2.
Every normal subgroup $\triangle$ of $\Ga$ of index $n$ defines a covering surface $\Sigma_r$ when $r = n+1$, for which $\Ga/\triangle$  serves as a group of (orientation preserving) isometries and hence define a finite subgroup of $\Homeo^+(\Sigma_r)$. Now, the number of isomorphism classes of finite groups of order at most $n$ which are generated by 2 elements is super polynomial (in fact, this number of groups grows like $n^{O (\log n)}$  -- see [Lu] and the references therein).  Thus, there is an infinite set of $r$'s for which there is an unbounded number of non isomorphic finite subgroups of $\Homeo^+(\Sigma_r)$ of order $r-1$.
Even if these subgroups are not maximal, there are also unboundedly many isomorphism classes of maximal subgroups.  Indeed, each one of the above is of index at most 84 (by Hurwitz upper bound) in a maximal subgroup.  Now, every maximal subgroup containing one of these subgroups is generated by at most $2 + \log_2 (84) < 9$ elements.  The number of bounded index subgroups in any group with a bounded number of generators is uniformly bounded.  This finishes the proof of the Proposition.
\end{proof}

\section{Topological rigidity}

In this section we prove Theorem 1.7 and 1.8.

Theorem 1.8  is proved essentially the same way as Theorem 1.5, but two  modifications are needed: Theorem 1.5 was proved by restricting to actions of the cyclic group $G = \bbz/p$, while now we need similar results for actions of general finite groups $G$.  Let us indicate how the method of proof for $\bbz/p$ generalizes to general $G$.
\begin{enumerate}
\item Note that if $H$ is any group acting semifreely (i.e.,  with only two kinds of orbits, fixed points and free orbits) the literally same proof as for $\bbz/p$  works.
\item Now if $G$ is a finite group, consider the least singular of the  singular points, i.e., the non-singular points of the singular set.  Each of these will be fixed by some group $H$.  The $H$-fixed set consists of points fixed by $H$, and maybe also some more singular points fixed by a larger group.  We will do our modification near points that are fixed only by $H$ or by a conjugate of it.  These points are the $G$ orbit of points fixed just by $H$.
    \item The modification will be done as follows. Start with a semi-free $H$-sphere $S$ with fixed set $\Sigma$ and normal representation - the $H$ representation that occurs at a fixed point of $H$  (that is on the top stratum of the singular set, as in \# 2).  We can consider the product space $(G \times S)/H$, where $H$ acts on the left on $G$ and the right on $S$. So,  $G$ acts on this product.  The underlying topological space is $G/H \times S$, but the action is more interesting.  It is called the induction of the $H$-action on $S$ to $G$.
        \item Now take connected sum along an orbit of $M$ with $(G \times S)/H$.  It is homeomorphic to $M$, but  the singular set is modified by connect sums of copies  of $\Sigma$ in various places.
            \item Similar tricks work when we do Edwards modifications.
            \end{enumerate}

            With the above modification all the results proved in Sections 2-5 can be modified to work with general finite group $G$.

            The second modification is easier:  We should think of proper discontinuous actions of $\Ga$ on $H/K$ as follows.  Let $\triangle$ be a normal finite index torsion free subgroup of $\Ga$.  Then $M = \triangle\setminus H/K$ is a compact manifold upon which $G = \Ga/\triangle$ acts.  Note that $M$ is indeed compact whatever the (proper discontinuous) action of $\Ga$ on $H/K$ is, since the cohomological  dimension of $\triangle$ is $\dim(H/K)$ as deduced from its original isometric action.  Now, taking the above mentioned modification (from $\bbz/p$ to $G$), Theorem 1.8 is deduced from Theorem 1.5 by standard covering space theory, changing from $M$ to its universal cover $H/K$.  Note, however,
 that the formulation of the two theorems is slightly different due to the fact that $\gamma \in \Ga$ has a fixed point in $H/K$ if and only if it is of finite order.  Also, an automorphism of $M=\triangle\setminus H/K$ with a fixed point can be lifted to an element of finite order  in $\Ga = N_H(\triangle)$.

 Theorem 1.7 is essentially equivalent to Corollary 1.6, but one needs to ensure that when there are countably many proper discontinuous actions of $\Ga$ on $H/K$, these actions are isolated, i.e., a small perturbation of each such action is conjugate to it.  This is indeed the case (in contrast with the theory of deformations into Lie groups in which case, if there are infinitely many cocompact discrete representations, then local rigidity fails and there are continuously many such actions.) The point is that when the fixed points form a discrete set one does have a local topological rigidity (and even in dimension 4).  This follows from Edmonds' Theorem \cite[Theorem 2.8]{Ed} in high dimensions.  The work of \cite{FrQ} (see also \cite{FeW}) shows it is true also in dimension 4.  The case of dimension 3 is always covered by either Theorem 1.5 (a) or (c), and case (b) does not happen.
 \hfill $\square$


\end{document}